\documentclass[a4paper, 11pt, draft]{amsart}
\usepackage{amsmath, amsthm, amscd, amssymb, amsfonts, amsxtra, amssymb, latexsym, bm}
\usepackage{enumerate}
\usepackage{verbatim}
\usepackage{tgpagella}
\usepackage{float}
\usepackage{inputenc, t1enc}

\newcommand{\Z}{\mathbb{Z}}
\newcommand{\N}{\mathbb{N}}
\newcommand{\ff}{\mathbb{F}}
\newcommand{\Tr}{\operatorname{Tr}}

\newcommand{\CC}{\mathcal{C}}

\newcommand{\G}{\Gamma}

\newcommand{\sk}{\smallskip}

\newtheorem{thm}{Theorem}[section]
\newtheorem{prop}[thm]{Proposition}

\newtheorem{coro}[thm]{Corollary}

\theoremstyle{definition}
\newtheorem{rem}[thm]{Remark}
\newtheorem{exam}[thm]{Example}

\theoremstyle{remark}

\usepackage{color}

\begin{document}
\numberwithin{equation}{section}
\title[Weight distribution of cyclic codes and decomposable GP-graphs]{The weight distribution of irreducible cyclic codes associated with decomposable generalized Paley graphs} 
\author{Ricardo A.\@ Podest\'a, Denis E.\@ Videla}
\dedicatory{\today}
\thanks{Partially supported by CONICET, FonCyT and SECyT-UNC}
\address{Ricardo A.\@ Podest\'a. FaMAF -- CIEM (CONICET), Universidad Nacional de C\'ordoba. 
	Av.\@ Medina Allende 2144, (5000) C\'ordoba, Rep\'ublica Argentina.
{\it E-mail: podesta@famaf.unc.edu.ar}}

\address{Denis E.\@ Videla.
	FaMAF -- CIEM (CONICET), Universidad Nacional de C\'ordoba. 
	Av.\@ Medina Allende 2144, (5000) C\'ordoba, Rep\'ublica Argentina. 
{\it E-mail: dvidela@famaf.unc.edu.ar}}

\begin{abstract}
We use known characterizations of generalized Paley graphs which are Cartesian decomposable to explicitly compute the spectra 
of the corresponding associated irreducible cyclic codes.  
As applications, we give reduction formulas for the number of rational points in Artin-Schreier curves defined over extension fields and to the computation of Gaussian periods.
\end{abstract}

\maketitle

\section{Introduction}
In a recent work \cite{PV3} (see also \cite{PV1}) we have related the spectra of generalized Paley graphs with the weight distribution of certain associated irreducible cyclic codes. In this work we will compute the weight distribution of irreducible cyclic codes whose associated generalized Paley graphs are Cartesian decomposable. We now recall the basic definitions and results about these codes and graphs. 

\subsubsection*{Generalized Paley graphs}
Let $k$ and $q$ be integers, with $q$ a prime power, say $q=p^m$. A \textit{generalized Paley graph} is a Cayley graph (\textit{GP-graph} for short) of the form
\begin{equation} \label{Gammas}
\G(k,q) = X(\ff_{q},R_{k}) \quad \text{with } \quad R_{k} = \{ x^{k} : x \in \ff_{q}^*\}.
\end{equation} 
That is, $\G(k,q)$ is the graph with vertex set $\ff_{q}$ and two vertices $u,v \in \ff_{q}$ are neighbors (directed edge) if and only if $v-u=x^k$ for some $x\in \ff_q^*$. 
Notice that if $\omega$ is a primitive element of $\ff_{q}$, then $R_{k} = \langle \omega^{k} \rangle = \langle \omega^{(k,q-1)} \rangle$. This implies that $\G(k,q)= \G((k,q-1),q)$
and that it is a $\frac{q-1}{(k,q-1)}$-regular graph. 
Thus, we will always assume that $k \mid q-1$. 
Although the graphs $\Gamma(k,q)$ were denoted as $GP(q,\frac{q-1}k)$ in \cite{LP}, our notation is more suited to our purposes because of the relation with the codes $\CC(k,q)$ that will be defined later.
The graph $\G(k,q)$ is undirected if $q$ is even or if $k \mid \tfrac{q-1}2$ for $p$ odd, and it is 
connected if the regularity degree 
$$n = \tfrac{q-1}{k}$$ 
is a primitive divisor of $q-1$ (see \cite{LP}).
For $k=1, 2$ we get the complete graph $\G(1,q)=K_q$ and the classic Paley graph $\Gamma(2,q) = P(q)$.

The spectrum $Spec(\G)$ of a graph $\G$ is the spectrum of its adjacency matrix. 
If $\Gamma$ has different eigenvalues $\lambda_0, \ldots, \lambda_t$ with multiplicities $m_0,\ldots,m_t$, we write 
as usual 
$$Spec(\Gamma) = \{[\lambda_0]^{m_0}, \ldots, [\lambda_t]^{m_t}\}$$
with $\lambda_0 > \cdots > \lambda_t$.
It is well-known that an $n$-regular graph $\G$ has $n$ as it biggest eigenvalue,
with multiplicity equal to the number of connected components of $\G$. 

There are few cases of explicitly known spectrum of Cayley graphs. For instance: unitary Cayley graphs over rings $X(R,R^*)$, where $R$ is a finite abelian ring and $R^*$ is the group of units (\cite{Ak}, this includes the classic unitary graphs $X(\Z_n,\Z_n^*)$) and generalized Paley graphs $X(\ff_{q^m},S_\ell)$ with $S_\ell=\{ x^{q^\ell+1} : x\in \ff_{q^m}^*\}$ where $\ell \mid m$ (\cite{PV1}, this includes the classic Paley graphs).

The eigenvalues of the GP-graphs $\Gamma(k,q)$, being Cayley graphs, are given by 
\begin{equation} \label{lambdagamma}
\lambda_\gamma=\sum_{y \in R_k} \chi_\gamma(y)
\end{equation}
for each $\gamma \in \ff_q$, where $\{\chi_\gamma\}$ are the irreducible characters of $\ff_q$.
In \cite{PV3}, we studied the spectrum of $\Gamma(k,q)$ and showed that 
these eigenvalues $\lambda_\gamma$
coincide with the Gaussian periods
\begin{equation} \label{gaussian}
\eta_{i}^{(N,q)} = \sum_{x\in C_{i}^{(N,q)}} e^{\frac{2\pi i}{p} {\Tr_{q/p}(x)}} \in \mathbb{C}, \qquad 0 \le i \le N-1,
\end{equation} 
where 
$C_{i}^{(N,q)} = \omega^{i} \, \langle \omega^N \rangle$
is the coset in $\ff_q$ of the subgroup $\langle \omega^{N} \rangle$ of $\ff_{q}^*$ and 
\begin{equation} \label{N} 
N = \gcd(\tfrac{q-1}{p-1}, k).
\end{equation} 
More explicitly, we showed that
\begin{equation*} \label{spec Gkq} 
Spec(\G(k,q)) = \{ [n]^{1+\mu n}, [\eta_{i_1}]^{\mu_{i_1} n}, \ldots, [\eta_{i_s}]^{\mu_{i_s} n} \}.
\end{equation*}
for some integers $\mu, \mu_{i_1}, \ldots, \mu_{i_s}$ (see Theorem 2.1 in \cite{PV3}).

\subsubsection*{Cartesian graph product} 
There are many different kind of products in graph theory. 
The most common ones are the tensor product (also called Kronecker product or direct product), the strong product and the Cartesian product (also called box product or sum of graphs). These products allow, in different contexts, to determine some graphs invariants such as: chromatic, clique and independence numbers, diameter, eigenvalues and energy, and also the automorphism group.
A complete study of these products can be found in \cite{HIK}. 
It is remarkable that they are particular cases of another graph operation called NEPS (non-complete extended $p$-sum, see \cite{CDS}). 
In this work, we will only deal with the Cartesian product.

The \textit{Cartesian product} of the graphs $\Gamma_1,\ldots,\Gamma_t$ with $t>1$, is the graph 
\begin{equation} \label{decomp}
\Gamma = \Gamma_1 \, \square \cdots \square\,  \Gamma_t,
\end{equation}
with vertex set $V(\Gamma) = V(\Gamma_1)\times\cdots\times V(\Gamma_t)$, such that $(v_1,\ldots,v_{t})$ and 
$(w_1,\ldots,w_t)$ in $V(\G)$ form an edge in $\Gamma$ if and only if 
there is only one $j\in\{1,\ldots,t\}$ such that $\{v_j,w_j\}$ is an edge in $\Gamma_j$ and 
$v_i=w_i$ for all $i\neq j$. For instance, $K_2 \square K_2=C_4$, the Cartesian product of $K_2$ and a path graph is a ladder graph and the Cartesian product of two path graphs is a grid graph. Also, the Cartesian product of $n$ edges is an $n$-hypercube $(K_2)^{\square n} =Q_n$, the Cartesian product of two hypercube graphs is another hypercube: $Q_n \square Q_m = Q_{n+m}$, and the Cartesian product of two complete graphs $K_n \square K_m$ is the $n \times m$ rook's graph.
Another important class of Cartesian product graphs is given by the Hamming graphs. A \textit{Hamming graph} $H(b,m)$ 
is any graph with vertex set all the $b$-tuples with entries from a set $V$ of size 
$m$, and two $b$-tuples form an edge if and only if they differ in exactly one coordinate.
Clearly, 
\begin{equation} \label{Hbm}
H(b,m)= (K_{m})^{\square b}
\end{equation}
for positive integers $b,m$ such that $b,m>1$. Notice that $Q_n=H(n,2)$.

It is a classic result of Sabidussi (\cite{Sa}) from 1957 that the chromatic number of the Cartesian product satisfies 
$\chi(G \square H) = \max \{ \chi(G), \chi(H) \}$. 
Hence, a Cartesian product is bipartite if and only if each of its factors is. 
More recently in 2000, Imrich and Klav\'zar proved that a Cartesian product is vertex transitive if and only if each of its factors is (\cite{IK}).
If a connected graph is a Cartesian product, it can be factorized uniquely as a product of prime factors; thay is, graphs that cannot themselves be decomposed as products of graphs (\cite{Sa2}).

\subsubsection*{Irreducible cyclic codes}
A linear code of length $n$ over $\ff_q$ is a vector subspace $\CC$ of $\ff_q^n$. The weight of a codeword 
$c=(c_{0},\ldots,c_{n-1}) \in \CC$ 
is the number $w(c)$ of its nonzero coordinates. 
The weight distribution of $\CC$, denoted 
$$Spec(\CC) = (A_0,\ldots,A_{n}),$$ 
is the sequence of frequencies 
$A_{i} = \#\{c\in\mathcal{C} : w(c)=i\}$. If $w_0=0 < w_1 < \cdots < w_t$ are the non-zero weights, then $\CC$ is called a $t$-\textit{weight} code and $w_1$ is the minimum distance of $\CC$.

A linear code $\CC$ is cyclic if for every $(c_{0},\ldots,c_{n-1})$ in $\CC$ the shifted codeword $(c_{1},\ldots,c_{n-1},c_{0})$ is also in $\CC$.  
An important subfamily of cyclic codes is given by the irreducible cyclic codes. 
For $k \mid q-1$ we will be concerned with the weight distribution of the $p$-ary irreducible cyclic codes 
\begin{equation} \label{Ckqs}
\mathcal{C}(k,q) = \big \{ c_{\gamma} = \big( \Tr_{q/p}(\gamma\, \omega^{ki}) \big)_{i=0}^{n_\CC-1} : \gamma \in \ff_{q} \big \} 
\end{equation}
where $\omega$ is a primitive element of $\ff_{q}$ over $\ff_p$ 
 and $\Tr_{q/p}$ denotes the trace map from $\ff_q$ to $\ff_p$. 
This code has zero $\omega^{-k}$ and length
\begin{equation} \label{N and n} 
n_\CC = \tfrac {q-1}N 
\end{equation}  
with $N$ as in \eqref{N}. Sometimes, we will need to further assume that
$k \mid \tfrac{q-1}{p-1}$.
This extra assumption implies that $N=k$ and $n_\CC=n=\frac{q-1}k$.

The computation of the spectrum of (irreducible) cyclic codes is in general a difficult task. 
There are several papers on the computation of the spectra of some of these cyclic codes by using exponential sums 
(see for instance \cite{BMc}, \cite{FL}, \cite{LYL}, \cite{Mc}). 
A complete survey on this topic can be found in \cite{DLY} (see also \cite{DY} for the irreducible case).
It is well-known that the weights of irreducible cyclic codes 
can be calculated in terms of Gaussian periods (see for instance \cite{BEW}, \cite{Di1}, \cite{Di2} and \cite{Mc}).
In fact, using Gaussian periods we have recently showed in \cite{PV3} that if $\G(k,q)$ is connected and $k\mid\frac{q-1}{p-1}$, then the eigenvalue $\lambda_{\gamma}$ of $\G(k,q)$ and the weight of $c_{\gamma} \in \CC(k,q)$ for $\gamma$ in $\ff_q$ are related by the simple expression
\begin{equation} \label{Pesoaut}
\lambda_{\gamma} = n - \tfrac{p}{p-1} w(c_{\gamma}) 
\end{equation}
where $n=\frac{q-1}k$ is both the length of $\CC(k,q)$ and the regularity degree of $\Gamma(k,q)$. Moreover, the frequency of the weight $w(c_{\gamma})$ coincides with the multiplicity of $\lambda_{\gamma}$.

We will assume henceforth that $q=p^m$ for some natural $m$ with $p$ prime and that $k$ is a positive integer such that $k \mid q-1$.  We next give a brief summary of the results of the paper.

\subsubsection*{Outline and results}
In Section 2, we consider the weight distribution of the code $\CC=\CC(k,q)$ associated with the graph $\G=\G(k,q)$ which is Cartesian decomposable. In this case it was proved by Pearce and Praeger (\cite{PP}) that $\Gamma = \square^b \G_0$ for some fixed GP-graph $\G_0$. In Theorem \ref{bebes2}, we show that the weight distribution of $\CC$ can be computed from the corresponding one of the smaller code $\CC_0$ associated with $\G_0$. 
In fact, the weights of $\CC$ are certain integral linear combinations $w = \ell_1 w_1 +\cdots + \ell_s w_s$	
of the weights of $\CC_0$.

In the next two sections we obtain the weight distributions of irreducible cyclic codes $\CC$ constructed from 1-weight and 2-weight irreducible cyclic codes. 
In the 1-weight case, the weight distribution of $\CC$ is obtained from the code $\CC(1,q)$, which in the binary case is the simplex code (see Proposition \eqref{1-weight}).
In the case of irreducible 2-weight cyclic codes, they are of three different kind: subfield, semiprimitive and exceptional. 
Subfield subcodes are not connected and hence cannot be considered by our methods. 
The semiprimitive case is studied in Proposition~\ref{bebes3}.
The computations of the exceptional cases can be done in the same way, but are left over because they are quite unmanageable.

Section 5 deals with the weight distribution of the irreducible cyclic codes constructed from the codes $\CC(3,q)$ and $\CC(4,q)$.
We have that $\G(1,q)$ and $\G(2,q)$ are the complete and the classic Paley graphs, respectively. The next graphs to consider are $\G(3,q)$ and $\G(4,q)$. The weight distributions of the associated codes $\CC(3,q)$ and $\CC(4,q)$ are known (see Theorems 19--21 in \cite{DY}).
In Theorems 5.1 and 5.3 we give the weight distributions of the irreducible cyclic codes constructed from $\CC(3,q)$ and $\CC(4,q)$, respectively.  

The final two sections are devoted to applications of the reduction result given in Theorem~2.2. 
In Section 6, we consider Artin-Schreier curves $C_{k,\beta}(p^m)$ with affine equations 
$$y^p-y=\beta x^k$$
where $k\mid \frac{p^m-1}{p-1}$ and $\beta\in \ff_{p^m}$.
In Proposition 6.1 we obtain the direct relationship 
$$\#C_{k,\beta}(p^m) = 2p^m+k(p-1) \lambda_\beta $$
between the the number of $\ff_{p^m}$-rational points of the curve $C_{k,\beta}(p^m)$ and the eigenvalues of the graph $\G(k,p^m)$. 
Then, via this connection between curves and graphs and the reduction result for graphs, 
we express the number of rational points of an Artin-Schreier curve over a field $\ff_{p^{ab}}$ in terms of linear combinations of the number of rational points of Artin-Schreier curves over a subfield $\ff_{p^a}$ (see Corollary~\ref{ratpoints}).

Finally, as a second application, we give an expression for Gaussian periods in terms of Gaussian periods of smaller parameters 
(Proposition \ref{gaussitos}). In fact, we show that, under the same hypothesis of Cartesian decomposability of GP-graphs used in Theorem 2.2, each Gaussian period $\eta_i^{(k,p^{ab})}$ is an integral linear combination of Gaussian periods $\eta_j^{(u,p^a)}$. In the particular case that the smaller pair is semiprimitive, we get the simple explicit expression of Proposition \ref{GP upr}.

\section{Spectrum of cyclic codes associated with decomposable GP-graphs}
A graph $\Gamma$ is \textit{Cartesian decomposable} if it can be written as a product of smaller graphs 
$\Gamma_1, \ldots, \Gamma_t$ as in \eqref{decomp} 
with $t>1$.
Recently, Pearce and Praeger (\cite{PP}) characterized those generalized Paley graphs which are Cartesian decomposable. 
They proved that a Cartesian decomposable GP-graph is a product of copies of a single graph, which is necessarily another GP-graph. 

More precisely, if $\Gamma = \G(k,p^m)$ is simple and connected, that is if $k$ divides $\tfrac{q-1}2$ when $p$ is odd 
and $n=\tfrac{p^m-1}k$ is a primitive divisor of $p^m-1$, 
the following conditions are equivalent:
\begin{equation} \label{cond desc}
\begin{split}
& (a) \quad \Gamma = \G(k,p^m) \text{ is Cartesian decomposable}. \\
&	(b) \quad \text{ $n=bc$ with $b>1$, $b\mid m$ and $c$ is a primitive divisor of $p^{\frac{m}{b}}-1$}. \\
&	(c) \quad \Gamma \cong \square^{b} \Gamma_0, \text{ where $\Gamma_{0}= \G(u, p^{\frac mb})$ with $u=\tfrac{p^{\frac mb}-1}{c}$ for $b,c$ as in $(b)$}. 
\end{split}
\end{equation}
We recall that $n$ is a \textit{primitive divisor} of $p^m-1$ if $n \mid p^m-1$ and $n \nmid p^t-1$ for all $t<m$. To denote this fact, for convenience, we will use the following notation   
\begin{equation} \label{prim div}
n \dagger p^m-1.
\end{equation}

We want to point out the following structural consequence of the previous result of Pearce and Praeger for those GP-graphs which are strongly regular.

\begin{prop} \label{cdsrg}
Let $\G(k,q)$ be a connected GP-graph which is strongly regular.
Then, $\G(k,q)$ is Cartesian decomposable if and only if it is a Hamming graph and $q$ is a perfect square. 
In this case we have
\begin{equation} \label{hammingGPsrg}
\G(k,q) = K_{q'} \square K_{q'} = H(2,q')
\end{equation}
with $k=\frac{\sqrt q + 1}{2}$ and $q'=\sqrt q$.
\end{prop}

\begin{proof}
Suppose that $\Gamma = \Gamma(k,q)$, with $q=p^m$, is a Cartesian decomposable strongly regular graph. 
Since $\G$ is connected, by \eqref{cond desc} we have
$\Gamma\cong \square^{b} \Gamma_0,$ where $\Gamma_0=\Gamma(u,p^a)$ with $c=\frac{p^a-1}{u}\dagger p^a-1$, $m=ab$ and 
$n=\tfrac{p^m-1}{k} = bc$. Also, 
since every connected strongly regular graph has only two non-trivial eigenvalues, necessarily $b=2$ and $\Gamma_0$ is a complete graph. Otherwise, $\Gamma$ would have more than two nontrivial eigenvalues because all of the eigenvalues of the Cartesian product of graphs are sums of eigenvalues of its factors. Thus, $\Gamma_0$ must have only two eigenvalues and $b=2$. But the graphs with only two eigenvalues are exactly disjoint unions of two copies of the same complete graph. As $c=\frac{p^a-1}{u}\dagger p^a-1$ then $\G_0$ is connected. Therefore, $\Gamma_0$ is the complete graph with $p^a$ vertices, since $b=2$, and then $q'=p^a=\sqrt{q}$. 
On the other hand, we have 
$$k=\tfrac{p^m-1}{n}=\tfrac{p^{2a}-1}{2(p^a-1)}=\tfrac{p^a+1}{2}=\tfrac{\sqrt{q}+1}{2},$$
as desired. 
The second equality in \eqref{hammingGPsrg} follows by \eqref{Hbm}.

The converse is clear from the fact that the eigenvalues of $K_{q'} \square K_{q'}$ are $2q'-2$, $q'-2$ and $-2$, with multiplicities $1$, $2(q'-1) $ and $(q'-1)^2$, respectively.
\end{proof}

We next show that if $\Gamma$ is a Cartesian decomposable GP-graph, say $\Gamma \cong \square^{b} \Gamma_0$, then the computation of the spectrum of the cyclic code $\CC$ associated with $\Gamma$ reduces to the one of the smaller code $\CC_0$ associated with $\Gamma_0$.
We will use a recent result in \cite{PV3} relating the spectra of $\CC_0$ with the one of $\Gamma_0$.

In the sequel we assume that $p$ is a prime and $k,m$ are positive integers such that 
\begin{equation} \label{conds 1}
q=p^m, \qquad k\mid q-1 \qquad \text{and put} \qquad n=\tfrac{q-1}{k}.
\end{equation}
In addition, sometimes we will also require that   
 \begin{equation} \label{conds 2}
 m=ab, \qquad n=bc \qquad \text{and} \qquad u=\tfrac{p^a-1}{c},
 \end{equation}
for non-negative integers $a,b,c$ with $b>1$ and $c\mid p^a-1$. Notice that if $p-1 \mid c$ (or, equivalently, if $u \mid \frac{p^a-1}{p-1}$), then $k\mid \frac{q-1}{p-1}$.

We are now in a position to state and prove our main result.

\begin{thm} \label{bebes2}
	Let $p, q, k, m, n, a, b, c, u$ be positive integers as in \eqref{conds 1}--\eqref{conds 2} with 
	$c \dagger p^a-1$ and $n \dagger p^m-1$. 
	Consider the irreducible cyclic codes $\CC = \CC(k, p^m)$ and $\CC_0 = \CC(u, p^{a})$. 
	If $p-1\mid c$, then $Spec(\CC)$ is determined by $Spec(\CC_0)$. 
	More precisely, if the weights of $\CC_0$ are $0=w_0<w_1 < \cdots <w_s$ with frequencies 
	$A_{w_i} = m_i$ for $i=0,\ldots,s$, then the weights of $\CC$ are given by
\begin{equation} \label{weights}
w_{\ell_0,\ldots,\ell_s} = \ell_1 w_1 + \cdots + \ell_s w_s
\end{equation}
	where $(\ell_0,\ldots,\ell_s) \in \N_0^{s+1}$ such that $\ell_0+\cdots+\ell_s=b$, with frequencies 
\begin{equation} \label{frequencies}
A_{\ell_0,\ldots,\ell_s} =  \tbinom{b}{\ell_0, \ldots, \ell_s} \, m_1^{\ell_1}\cdots m_s^{\ell_s}.
\end{equation}
In particular, $\CC$ has the same minimum distance as $\CC_0$.
\end{thm}

\begin{proof}
	We have $k\mid \frac{p^m-1}{p-1}$ and $u \mid \frac{p^{a}-1}{p-1}$ since $p-1\mid c$. 
	Also, the graphs $\G = \G(k,p^m)$ and $\G_0=\G(u,p^{a})$ are connected because of the primitiveness of $n$ and $c$, respectively. Thus, we can apply Theorem 5.1 in \cite{PV3} to the codes $\CC$ and $\CC_0$ 
	of lengths $n$ and $c$, respectively.
	
	By hypothesis, since conditions \eqref{cond desc} are satisfied, we have  
	that $\Gamma\cong \square^{b} \Gamma_0$ and therefore 
	$\mathrm{Spec}(\G) = \mathrm{Spec}(\square^{b} \Gamma_0)$.
	It is known that the eigenvalues of the 
	Cartesian product of graphs is the sum of the eigenvalues of its factors (see for instance \cite{CDS}). 
	Now, if $\mathrm{Spec}(\Gamma_0)=\{[\lambda_0]^{m_0},[\lambda_1]^{m_1},\ldots,[\lambda_s]^{m_s}\}$ where $\lambda_0=c$ is the trivial eigenvalue with multiplicity $m_0=1$, then the eigenvalues of $\Gamma$ are 
	\begin{equation} \label{spec sum} 
	\Lambda_{\ell_0,\ldots,\ell_s} = \ell_0 \lambda_0+ \cdots + \ell_s \lambda_s 
	\end{equation}
	where the $(s+1)$-tuple of integers $(\ell_0,\ldots,\ell_s)$ satisfies 
	$$\ell_0+\cdots+\ell_s=b$$ and 
	$\ell_i\ge 0$ for every $i$, 
	with corresponding multiplicity 
	\begin{equation} \label{spec sum mult} 
	\tbinom{b}{\ell_0, \ldots, \ell_s} m_1^{\ell_1} \cdots m_s^{\ell_s}
	\end{equation}
since $m_0=1$, where 
	$\tbinom{b}{\ell_0, \ldots, \ell_s}$ stands for the multinomial coefficient. 
	The hypothesis $p-1\mid c$ is equivalent to $u \mid \frac{p^a-1}{p-1}$. Thus, we can apply Theorem 5.1 in \cite{PV3} 
	(i.e.\@ \eqref{Pesoaut})  to the graph $\Gamma_0$ and the code $\CC_0$ and, hence, we have
	$$\lambda_i = c - \tfrac{p}{p-1} w_i$$ 
	for each $i=0,\ldots, s$. Therefore, we get 
	$$\Lambda_{\ell_0,\ldots,\ell_s} = \sum_{i=0}^s \ell_i(c-\tfrac{p}{p-1}w_i) = n - \tfrac{p}{p-1} \sum_{i=1}^s \ell_i w_i,$$ 
	since $c(\ell_0+\cdots+\ell_s) = bc=n$ and $w_0=0$. 
	Also, the frequency of $w_i$ in $\CC_0$ is $m_i$ for all $i=0,\ldots,s$. 
	Since $p-1\mid c$ we have $k \mid \frac{q-1}{p-1}$ and hence by \eqref{Pesoaut}  
	again applied to $\CC$ and $\Gamma$, we have that the weights of $\CC$ are given by
	$$w_{\ell_0,\ldots,\ell_s} = \tfrac{p-1}{p}(n-\Lambda_{\ell_0,\ldots,\ell_s}) = \ell_1 w_1 + \cdots + \ell_s w_s$$ 
	with frequencies $\binom{b}{\ell_0, \ldots, \ell_s}m_1^{\ell_1}\cdots m_s^{\ell_s}$, as desired.
	
	The last assertion is straightforward from \eqref{weights}, and the result follows.
\end{proof}

\begin{rem}
	Suppose that $\Gamma \simeq \square^b \Gamma_0$. Then $\Gamma$ and $\Gamma_0$ have associated irreducible cyclic codes $\CC$ and $\CC_0$.
	Under the hypothesis of the theorem, we can only assure that the spectrum of $\CC$ equals the spectrum of the direct sum code 
	$$\CC_0^b = \CC_0 \oplus \cdots \oplus \CC_0,$$
	 with $\CC_0$ repeated $b$-times, which is not cyclic in general.
	Thus, one may wonder if there is some code operation $*$ such that 
	$\CC = \CC_0 * \dots * \CC_0$,
	with $\CC_0$ repeated $b$-times. 
\end{rem}

\section{Cyclic codes from 1-weight cyclic codes}
In this and the next two sections we will apply Theorem \ref{bebes2} to compute the spectra of irreducible cyclic codes constructed from irreducible cyclic codes with few weights. We consider 1-weight irreducible cyclic codes here and 2-weight irreducible cyclic codes in the next section. In Section 5 we will deal with some codes that are 3-weight and 4-weight irreducible cyclic codes.

One-weight irreducible cyclic codes are already characterized when $k\mid q-1$ (see \cite{DY}, \cite{VW}). 
In fact, by Theorem 16 in \cite{DY}, we have that if $k\mid q-1$, then the cyclic code $\CC(k,q)$ is irreducible if and only if 
$$N=\gcd(\tfrac{q-1}{p-1},k)=1.$$
In our case, the restriction $k\mid \frac{q-1}{p-1}$ implies that $k=1$ and hence, the only irreducible cyclic code that we can take into account is 
$$\CC(1,q) = \{ (\Tr_{q/p}(\gamma \omega^i))_{i=0}^{q-1} : \gamma \in \ff_q\} $$ 
over $\ff_{p}$ of length $q-1$, where $\omega$ is a primitive element of $\ff_q$. Note that in the binary case $p=2$, $\CC(1,2^m)$ is just the simplex code (i.e.\@ the dual of the Hamming code).

From now on, it will be useful to use the following notation
\begin{equation} \label{psi b}
\Psi_{b}(x) = \tfrac{x^b-1}{x-1} = x^{b-1} + \cdots +x^2 + x +1 .
\end{equation}

\begin{prop} \label{1-weight}
	Let $q=p^a$ with $p$ prime, $a\ge 1$ and $b>1$ an integer dividing $\Psi_{b}(q)$. Put 
	$k = k_b = \tfrac 1b \Psi_{b}(q)$. 
	Then, $\CC= \CC(k, q^b)$ is an irreducible $b$-weight cyclic code with weights 
	$0,w,2w,\ldots,bw$ and frequencies given by 
	\begin{equation}\label{Spec C}
	Spec(\CC)= \{ A_{\ell w}(\CC) = \tbinom{b}{\ell} {A_{w}}^{\ell} \}_{0\le \ell \le b}
	\end{equation}
	where $w=(p-1) p^{a-1}$ and $A_w=p^{a}-1$ is the weight distribution of the code 
	$\CC_0 = \CC(1,q)$.
\end{prop}

\begin{proof}
	Notice that $k=\frac{q^b-1}{b(q-1)}$ and thus $k\mid\frac{q^b-1}{p-1}$. Clearly $c=q-1$ is a primitive divisor of itself. By 
	Theorem \ref{bebes2}, the spectrum of the irreducible cyclic code $\CC$ is determined by the spectra of the code 
	$\CC_0=\CC(1,q)$ if $n = \frac{q^b-1}{k} = b(q-1)$ is a primitive divisor of $q^b-1$. Equivalently, if the GP-graph $\G(k,q^b)$ is connected.

	Thus, we will show that $\G(k,q^b)$ is connected by showing that it is a Hamming graph. 
	For integers $d,q$ with $d>1$ and $q>1$, recall that a Hamming graph $H(d,q)$ is any graph with vertex set all the $d$-tuples with entries from a set $V$ of size 
	$q$, and two $d$-tuples form an edge if and only if they differ in exactly one coordinate. 
	
	In \cite{LP}, Lim and Praeger characterized all the GP-graphs which are Hamming graphs. 
	They proved that $\G(\frac{p^m-1}{n},p^m)$ is Hamming if and only if $n=b(p^{\frac mb}-1)$ for some divisor $b>1$ of $m$. 
	Clearly, $n$ satisfies this last condition and then $\G(k,q^b)$ is a Hamming graph, which is connected by definition. 
	This implies that $n$ is a primitive divisor of $q^b-1$. 
	
  We have that $\CC_0 = \CC(1,p^a)$. By Theorems 15 and 16 in \cite{DY}, $\CC_0$ is a 1-weight $[p^a, a, (p-1)p^{a-1}]$-code 
 	with weight distribution given by $w = (p-1)p^{a-1}$ with $A_w = p^a-1$. 
	The proposition thus follows from Theorem \ref{bebes2}.

By \eqref{weights} the weights are $w_{\ell_0,\ell_1} = \ell_1 w_1$ with $(\ell_0,\ell_1) \in \N_0^2$ such that $\ell_0+\ell_1=b$. 
Thus, the weights are 
$$w_\ell=\ell w \qquad \text{with} \qquad 0\le \ell \le b.$$
By \eqref{frequencies} the frequencies are given by 
$$A_{\ell w} = A_{\ell_0, \ell_1} = \tbinom{b}{\ell_0, \ell_1} m_1^{\ell_1}= \tfrac{b!}{\ell_0! \ell_1!} m_1^{\ell_1} =
\tbinom{b}{\ell_1} A_w^{\ell_1}.$$
Since $\ell_1$ runs from $0$ to $b$, we get the desired result. 
\end{proof}

Notice that one can check that \eqref{Spec C} is correct by adding the frequencies
$$\sum_{0\le \ell \le b} \tbinom{b}{\ell} A_w^\ell = (1+A_w)^b=p^{ab}=q^b.$$

\begin{exam}
Consider $p=2$ and $a=3$, hence $q=8$. One can check that if $b=7$ then $b \mid 8^6 + \cdots + 8^2 + 8 + 1 = 299{.}593$. 
The simplex code $\CC_0=\CC(1,8)$ has weights $w_0=0$, $w_1=4$, with frequencies $A_0=1$, $A_4=7$. 
Now, $k_b=\tfrac{1}{7} \Psi_7(8)=42{.}799$. By the previous proposition, the irreducible cyclic code 
$$\CC(\tfrac{1}{7} \Psi_7(8), 8^7) = \CC(42{.}799, 2{.}097{.}152)$$ 
has weight distribution
$$w_0=0, \:\: w_1=4, \:\: w_2=8, \:\: w_3=12, \:\: w_4=16, \:\: w_5=20, \:\: w_6=24, \:\: w_7=28$$
with frequencies 
\begin{align*}
& A_0=1, \quad & & A_4 = \tbinom 71 7 =7^2 = 49, \quad \\[1mm]  
& A_8 = \tbinom 72 7^2 = 3\cdot 7^3=1{.}029, \quad && A_{12} = \tbinom 73 7^3 = 5\cdot 7^4 = 12{.}005, \\[1mm] 
& A_{16} = \tbinom 74 7^4 = 5\cdot 7^5 = 84{.}035, \quad && A_{20} = \tbinom 75 7^5 = 3\cdot 7^6 = 352{.}947, \quad \\[1mm] 
& A_{24} = \tbinom 76 7^6 = 7^7 = 823{.}543 , \quad && A_{28} = 7^7 = 823{.}543.
\end{align*}
\end{exam}

By Theorem 17 in \cite{DY}, if $\frac{q-1}{p-1}$ is even, then $\CC(2,q)$ is a 2-weight irreducible cyclic code with non-zero weights $w^{\pm}= \frac{(p-1)(q\pm \sqrt q)}{qN}$ with frequencies $A_{w^{\pm}}=\frac{q-1}2$. The next result exhibits another infinite family of $2$-weight irreducible cyclic codes.

\begin{coro}
If $q$ is a power of an odd prime $p$, then $\CC(\frac{q+1}2, q^2)$ is a $2$-weight irreducible cyclic code
with weights $0, w, 2w$ with corresponding frequencies $A_0=1$, $A_w=1(q-1)$, $A_{2w} = (q-1)^2$ where $w=\tfrac{p-1}p q$.
\end{coro}

\begin{proof}
	Taking $b=2$, we clearly have that $2\mid \Psi_{2}(q)=q+1$ since $q$ is odd. Thus, thhe statement follows directly from Proposition \ref{1-weight}.
\end{proof}
\begin{rem}
	The code $\CC(\frac{q+1}2, q^2)$ belongs to the class of semiprimitive $2$-weight irreducible cyclic codes. In the following section, we will use this kind of codes to find other weight distributions.
	The GP-graph $\G$ associated with this code is the one given in Proposition \ref{cdsrg}, that is $\G=K_q \square K_q 
	= H(2,q)$. 
\end{rem}

By Proposition \ref{1-weight}, to get the weight distribution of $\CC(\frac{1}{b}\Psi_{b}(q),q^b)$  we only need to check 
that $b\mid \Psi_{b}(q)$. In the next result we give some sufficient conditions for this to happen, 
based on previous results on \cite{PV4}. 

\begin{coro}
Let $p$ be a prime and let $a,b,k,m,x$ be positive integers such that $m=ab$ with $b>1$, $k=\frac{1}{b}\Psi_{b}(p^a)$ and  $x=p^{a}$. The weight distribution of $\CC(k,p^m)$ is given by \eqref{Spec C} in the following cases:
\begin{enumerate}[$(a)$]
	\item If $b=r$ is a prime different from $p$ 
	and $x\equiv 1 \pmod r$. \sk 
	
	\item If $b=2r$ with $r$ an odd prime,  $x$ coprime with $b$ and $x\equiv \pm1 \pmod r$. \sk 
	
	\item If $b=r r'$ with $r<r'$ odd primes such that $r \nmid r'-1$ and $x\equiv 1 \pmod{rr'}$. \sk 
	
	\item If $b=r_1 r_2 \cdots r_\ell$ with $r_1 < r_2 < \cdots < r_\ell$ primes different from $p$ with $x\equiv 1 \pmod{r_1}$ and $x^{b/r_i} \equiv 1 \pmod{r_i}$ for $i=2, \ldots,\ell$. \sk

	\item If $b=r^t$ with $r$ prime such that $ord_{b}(x)=r^h$ for some $0\le h<t$. \sk 
	
	\item If $b = r_1^{t_1}	\cdots r_\ell^{t_\ell}$ with $r_1 < \cdots < r_\ell$ primes different from $p$ where $ord_{r_{i}^{t_i}}(x)=r_{i}^{h_i}$ with  $0\le h_i\le t_{i}-1$ for all $i$.
\end{enumerate}
\end{coro}

\begin{proof}
	Clearly ($a$)--($d$) are direct consequences of Proposition~\ref{1-weight} and the divisibility properties of $\Psi_{b}(x)$ in the square-free case given in Lemma 5.1 of \cite{PV4}. 
	On the other hand, ($e$) follows from Proposition~\ref{1-weight} and Lemma 5.2 of \cite{PV4}. 
	The remaining assertion is straightforward from Proposition~\ref{1-weight} and 
	Lemma 5.3 of \cite{PV4}.
\end{proof}

\section{Cyclic codes from $2$-weights cyclic codes}
In \cite{SW}, Schmidt and White conjectured that all two-weight irreducible cyclic codes of length $n$ over $\ff_p$, with $p-1\mid n$, belong to one of the following disjoint families:

$\bullet$ The \textit{semiprimitive codes}, which are  
those $\CC(u,p^{a})$ such that $-1$ is a power of $p$ modulo $u$. Equivalently, $(k,q)$ is a \textit{semiprimitive pair}, that is 
$k \mid p^{t}+1$ for some $t$ such that $t \mid m$ and $m_{t} = \frac{m}{t}$ even, and $k \ne p^{\frac m2}+1$. 

$\bullet$ The \textit{subfield subcodes}, corresponding to  
$\CC(u,p^a)$ where $u=\frac{p^a-1}{p^{t}-1}$ with $t<a$. 

$\bullet$ The \textit{exceptional codes}, i.e.\@ irreducible 2-weight cyclic codes which are neither subfield subcodes nor semiprimitive codes.

If one does not require the condition $p-1\mid n$, Pinnawala and Rao (\cite{RP}) has given a family of 2-weight irreducible cyclic codes which are not of the previous kind.  

Notice that in the subfield subcode case, the graph $\G(u,p^a)$ is not connected since $c=\frac{p^a-1}u$ is not a primitive divisor 
of $p^a-1$; and thus we cannot apply Theorem \ref{bebes2}. 
Hence, we are only interested in the other two cases.

We now compute the spectrum of the code $\CC$ associated with the decomposable graph $\G \simeq \square^b \G_0$, 
where $\G_0$ is a semiprimitive GP-graph.

\begin{prop}
	\label{bebes3}
	Let $p,q,k,m,n,a,b,c,u$ be positive integers as in \eqref{conds 1}--\eqref{conds 2} 
	such that $n \dagger p^m-1$.  
	If $(u,p^a)$ is a semiprimitive pair then the weights of the code $\CC=\CC(k,p^m)$ are given by  
	\begin{equation} \label{peso semip} 
	w_{\ell_1, \ell_2} = \tfrac{(p-1)p^{\frac a2 -1}}{u} \big\{ \ell_1(p^{\frac a2}-\sigma(u-1)) + 
	\ell_2 \big( p^{\frac a2} + \sigma \big) \big\}
	\end{equation}
	for every pair of non-negative integers $\ell_1,\ell_2$ such that $0 \le \ell_1+\ell_2 \le b$,
	where we put $\sigma=\pm 1$ if $u\mid p^{\frac a2}\pm 1$, 
	with frequencies 
	\begin{equation} \label{freq semip} 
	A_{\ell_1, \ell_2} = \tbinom{b}{\ell_1}\tbinom{b-\ell_1}{\ell_2}c^{\ell_1+\ell_2} (u-1)^{\ell_1}.
	\end{equation}
\end{prop}

\begin{proof}
	Consider the semiprimitive irreducible cyclic code $\CC_0=\CC(u,p^a)$.
	Thus, we have that $u\mid p^{\ell}+1$ for some $\ell\mid a$ with $\frac{a}{\ell}$ even and that $\CC_0$ is a $2$-weight  code. By Remark~5.6 in \cite{PV3} the weights of $\CC_0$ are 
	$$w_1=\tfrac{(p-1)p^{\frac a2 -1}}{u}(p^{\frac a2}-\sigma(u-1)) \qquad \text{and} \qquad 
		w_2= \tfrac{(p-1)p^{\frac a2 -1}}{u}(p^{\frac a2}+\sigma)$$
	where $\sigma=(-1)^{\frac{m}{2\ell}+1}$ with $\ell$ the minimal positive integer such that $u\mid p^{\ell}+1$. 
	
	Since $(u,p^a)$ is a semiprimitive pair then 
	$u\mid\frac{p^a-1}{p-1}$ and $k\mid \frac{p^m-1}{p-1}$. 
	Indeed, assume that $u\mid p^{\ell}+1$ with $\frac{a}{\ell}$ even, then if we denote by $v=v_2(\frac{a}{\ell})$ the $2$-adic value of $m\ell$ we obtain that $\frac{a}{2^v}=h\ell$ for some $h$ odd. On the first hand, by taking into account that $p^\ell\equiv -1 \pmod{p^{\ell}+1}$ we obtain that 
	$$p^{\frac{a}{2^v}}=p^{h\ell}\equiv (-1)^h\equiv -1 \pmod{p^{\ell}+1}$$ 
	i.e\@  we have that $p^{\ell}+1\mid p^{\frac{a}{2^v}}+1$ and thus $u\mid p^{\frac{a}{2^v}}+1$. On the other hand, it is easy to see that 
	$$p^{a}-1=(p^{\frac{a}{2^v}}-1)\prod_{j=1}^{v}(p^{\frac{a}{2^j}}+1).$$ 
	Notice that $p-1\mid p^{\frac{a}{2^v}}-1$ and therefore $u\mid\frac{p^a-1}{p-1}$, as desired. 
	Now, since $u\mid \frac{p^a-1}{p-1}$ and $\Psi_{b}(p^a)=\frac{p^{ab}-1}{p^a-1}$ then $u\Psi_{b}(p^a)\mid \frac{p^{ab}-1}{p-1}$.
	Using that $k=\frac{u}{b}\Psi_{b}(p^a)$ we obtain that $k\mid\frac{p^m-1}{p-1}$, as we wanted.  
	Thus, by Remark~5.6 in \cite{PV3} the frequencies of $w_1,w_2$ are $m_1=c$ and $m_2=c(u-1)$, respectively.  
	
	Now, by hypothesis we have that $m=ab$ and $n=bc$ is a primitive divisor of $p^m-1$. Hence, by Theorem \ref{bebes2} 
	the weights of $\CC(k,p^m)$ are $w_{\ell_0,\ell_1,\ell_2} = \ell_1 w_1 + \ell_2 w_2$ where $(\ell_0,\ell_1,\ell_2)\in \N_0^3$ with $\ell_0+\ell_1+\ell_2=b$, with frequencies 
$$A_{w_{\ell_0,\ell_1,\ell_2}} = \tbinom{b}{\ell_0,\ell_1,\ell_2} m_1^{\ell_1}m_2^{\ell_2}
= \frac{b!}{\ell_1 ! \ell_2 ! (b-(\ell_1+\ell_2))!}m_1^{\ell_1} m_{2}^{\ell_2}.$$  
	Thus, disregarding $\ell_0$ we have that
	$$w_{\ell_1,\ell_2} = \ell_1 w_1 + \ell_2 w_2 = \tfrac{(p-1)p^{\frac a2 -1}}{u} \big\{ \ell_1(p^{\frac a2}-\sigma(u-1)) + 
	\ell_2 \big( p^{\frac a2} + \sigma \big) \big\}$$
with corresponding multiplicities
$$A_{w_{\ell_1,\ell_2}} = \tbinom{b}{\ell_1} \tbinom{b-\ell_1}{\ell_2}  c^{\ell_1+\ell_2} (u-1)^{\ell_1},$$
where $(\ell_1,\ell_2)$ runs over all $2$-tuples of non-negative integers such that $0\le \ell_1+\ell_2\le b$, and therefore we obtain \eqref{peso semip} and \eqref{freq semip}, as we wanted.		
\end{proof}

The proposition implies that one knows the weight distribution of the cyclic code $\CC=\CC(k,p^{ab})$ associated to the decomposable graph $\G(k,p^{ab})=\square^{b} \G_0(u,p^a)$ without need to know the weight distribution of the smaller cyclic code $\CC_0=\CC(u,p^a)$ associated to $\Gamma_0(u,p^a)$.

\begin{exam}
	Let $p$ be an odd prime and take $a=u=2$ and $b=3$. The graph $\G_0=\G(2, p^2)$ is the classic Paley graph over $\ff_{p^2}$ 
	with spectrum 
	$$Spec(\G_0) = \{[\tfrac{p^2-1}2]^1, [\tfrac{p-1}{2}]^{\frac{p^2-1}2}, [-\tfrac{p+1}{2}]^{\frac{p^2-1}2} \}$$ 
	(see for instance {\cite{GR}}). 
	Then, the two nonzero weights of the code $\CC_{0}=\CC(2,p^2)$, that can be obtained from \eqref{Pesoaut}, have multiplicity $\frac{p^2-1}2$. 
	We have $m=ab=6$ and $c=\frac{p^{2}-1}2$ and thus 
	$$n=bc=\tfrac{3(p^2-1)}2.$$ 
	Clearly, $c$ is a primitive divisor of $p^2-1$. 
	Notice that if $p\neq 3$ ($p=2t+1$ prime and $t \not \equiv 1 \pmod 3$), then $9\mid n$. In particular, if $p \equiv 2, 5, 7 \pmod 9$, then $n$ is a primitive divisor of $p^6-1$, since in these cases the order of $p$ modulo $9$ is $6$ 
	and then $9$ does not divide $p^{a}-1$ when $1\le a <6$. This implies that $n$ does not divide $p^{a}-1$, either. 
	
	In this case one can choose $\sigma=1$ or $-1$ indistinctly in the formula \eqref{peso semip}. 
	Therefore the code 
	$$\CC = \CC(\tfrac{2(p^6-1)}{3(p^2-1)},p^6) = \CC(\tfrac 23 (p^4+p^2+1),p^6)$$ 
	has weights 
	$$w_{\ell_1, \ell_2} = \tfrac{(p-1)}{2} \{ \ell_1(p-1) + \ell_2 (p + 1) \} = \tfrac{(p-1)^2}2 \ell_{1} + c\ell_{2}$$
	for every pair $0 \le \ell_1+\ell_2 \le 3$, with frequencies 
	$$A_{\ell_1, \ell_2} = \tbinom{3}{\ell_1}\tbinom{3-\ell_1}{\ell_2}({\tfrac{p^2-1}2})^{\ell_1+\ell_2}.$$
	
	If $\ell_{2}=0$, then $w_{1,0} = \tfrac{(p-1)^2}2$, $w_{2,0}=(p-1)^2$, and $w_{3,0}=\tfrac{3(p-1)^2}2$.
	If $\ell_{1}=0$, then $w_{0,1}=\tfrac{p^2-1}2$, $w_{0,2}=p^2-1$, and $w_{0,3}=\tfrac{3(p^2-1)}2$.
	Also, if $\ell_1$ and $\ell_{2}$ are nonzero, then  
	$w_{1,1}=p(p-1)$, $w_{2,1}=\tfrac{(p-1)}2 (3p-1)$ and $w_{1,2}=\tfrac{(p-1)}2 (3p+1)$. 
	One can check that if $p\neq 5$, all these weights are different and hence the spectrum of $\CC$ is given by 
	Table~\ref{specejemplito}. 
	\renewcommand{\arraystretch}{1.2}
	\begin{table}[h!]
		\centering
		\caption{Weight distribution of $\CC$ with $p\equiv 2,5,7\pmod 9$ and $p>5$.} \label{specejemplito}
		\begin{tabular}{|c|c|}  
			\hline 
			weight & frequency \\ 
			\hline 
			$w_{0,0} = 0$ & $A_{0,0}=1$ \\
			\hline
			$w_{1,0}=\tfrac{(p-1)^2}2$& $A_{1,0}=3(\tfrac{p^{2}-1}2)$  \\ 
			\hline
			$w_{2,0}=(p-1)^2$ & $A_{2,0}=3(\tfrac{p^{2}-1}2)^{2}$  \\
			\hline
			$w_{3,0}=\tfrac{3(p-1)^2}2$& $A_{3,0}=(\tfrac{p^{2}-1}2)^3$  \\
			\hline
			$w_{0,1}=\tfrac{p^{2}-1}2$& $A_{0,1}=3(\tfrac{p^{2}-1}2)$  \\ 
			\hline
		\end{tabular} \:\:
		\begin{tabular}{|c|c|}
			\hline 
			weight & frequency \\ 
			\hline 
			$w_{0,2}=p^2-1$ & $A_{0,2}=3(\tfrac{p^{2}-1}2)^{2}$  \\
			\hline
			$w_{0,3}=\tfrac{3(p^{2}-1)}2$& $A_{0,3}=(\tfrac{p^{2}-1}2)^3$  \\ 
			\hline  
			$w_{1,1}=p(p-1)$& $A_{1,1}=6(\tfrac{p^{2}-1}2)^{2}$ \\
			\hline  
			$w_{2,1}=\tfrac{p-1}{2} (3p-1)$& $A_{2,1}=3(\tfrac{p^{2}-1}2)^3$ \\
			\hline
			$w_{1,2}=\tfrac{p-1}{2} (3p+1)$& $A_{1,2}=3(\tfrac{p^{2}-1}2)^3$ \\
			\hline
		\end{tabular}
	\end{table}
	
	Notice that adding all the frequencies we get 
	$$\sum_{0\le i+j\le 3} A_{i,j} = 1+6c+12c^2+8 c^3=p^6$$ 
	and therefore the code $\CC$ has dimension $6$ and minimum distance 
	$\frac{(p-1)^2}{2}$. That is, $\CC$ has parameters 
	$[\tfrac{3(p^2-1)}2,6,\tfrac{(p-1)^2}2]$.

	For instance, if $p=5$, we have $\CC=\CC(\tfrac 23 (5^4+5^2+1), 5^6)= \CC(434, 15{.}625)$ with parameters 
	$[36,6,8]$ defined over $\ff_{5}$.	
	The weights of $\CC$ are given by
	\begin{gather*}
	w_{1,0}=8, \:\:w_{2,0}= 16, \:\:w_{3,0}=w_{0,2}=24, \:\:w_{0,1}= 12, \\[1mm]
	w_{0,3}=36, \:\:w_{1,1}=20, \:\:w_{2,1}=28, \:\:w_{1,2}=32,
		\end{gather*}
		with frequencies 
	\begin{gather*}
	A_{8}=A_{12}= 3c = 36, \quad A_{16}=3 c^{2}=432, \quad A_{20}=6c^2=864, \\
	A_{24}=(c+3)c^2=2{.}160, \quad A_{28}=A_{32}=3c^3 = 5{.}184, \quad A_{36}=c^3 = 1{.}728.
	\end{gather*}
	since $c=12$. 
	\hfill $\lozenge$
\end{exam}

\begin{rem}
The weight distribution of irreducible cyclic codes constructed from exceptional 2-weight irreducible cyclic codes 
can be obtained from Theorem~\ref{bebes2} and from the spectrum of the associated GP-graphs, which are computed in \cite{PV3}.
\end{rem}

\section{Cyclic codes from $\CC(3,q)$ and $\CC(4,q)$}
In general, $3$-weight or $4$-weight irreducible cyclic codes are not classified. In this section we will use the irreducible cyclic codes $\CC(3,q)$ and $\CC(4,q)$ to find new weight distributions of irreducible cyclic codes via the reduction formula obtained in Section~2. More precisely, for $u=3,4$, if $\CC_0=\CC(u,q)$  is the code associated to $\Gamma_0(u,q)$, we will compute the weight distributions of codes $\CC(k,q^r)$ associated to the Cartesian product graph $\Gamma(k,q^r) = \square^r \Gamma_0(u,q)$.

We begin with cyclic codes constructed from $\CC(3,q)$.
\begin{thm} \label{c3}
	Let $p$ and $r$ be different primes with $p\equiv 1 \pmod{3}$ and let $c,k,m,q,t$ be integers such that $m=3t$, $q=p^m$, $c=\frac{q-1}{3}$ and $k=\frac{3}{r}\Psi_{r}(q)$. If $q\equiv 1\pmod{r}$ and $(3,r)=1$, then the weights of $\CC(k,q^r)$ are given by 
	$$w_{\ell_1,\ell_2,\ell_3} = \tfrac{p-1}{3p} \{ hq+(a(\tfrac{\ell_2+\ell_3}{2}-\ell_1)+\tfrac{9b}{2}(\ell_2-\ell_3))p^t \}$$ 
	where $(\ell_1,\ell_2,\ell_3)\in \mathbb{N}_0^3$ such that $0\le h=\ell_1+\ell_2+\ell_3\le r$,
	 and $a,b$ are the unique integers satisfying 
	 $$4p^t=a^2+27b^2, \qquad  a \equiv 1\pmod{3} \qquad \text{ and } \qquad (a,p)=1,$$ 
	 with corresponding frequencies 
	 $$A_{\ell_1,\ell_2,\ell_3} = \tbinom{r}{h} \tbinom{h}{\ell_1,\ell_2,\ell_3} c^h.$$ 
\end{thm}

\begin{proof}
	The spectrum of $\CC(3,q)$ is given in Theorems 19 and 20 in  
	\cite{DY}, with different notations ($r$ for our $q$, $N$ for our $k$, etc).

	If $p\equiv 1 \pmod 3$, by Theorem 19 in \cite{DY}, the four weights of $\CC(3,q)$ are $w_0=0$,
	\begin{equation} \label{sp1} 
	w_1= \tfrac{(p-1)(q-a\sqrt[3]{q})}{3p}, \:\:w_2= \tfrac{(p-1)(q+\frac 12 (a+9b) \sqrt[3]{q})}{3p}, 
	\:\:w_3= \tfrac{(p-1)(q+\frac 12 (a-9b) \sqrt[3]{q})}{3p},
	\end{equation}
	with frequencies $A_0=1$ and $A_1=A_2=A_3=\frac{q-1}3=c$; 
	where $a$ and $b$ are the only integers satisfying $4\sqrt[3]{q}=a^2+27b^2$, $a\equiv 1 \pmod 3$ and $(a,p)=1$. 
	Clearly, $3\mid\frac{q-1}{p-1}$, since $p\equiv1 \pmod{3}$ and $m=3t$. Moreover, $c$ is a primitive divisor of $p^m-1$, since the associated graph $\G(3,p^m)$ is connected.
	
	Assume now that $(3,r)=1$ and $q\equiv 1\pmod{r}$, we will show now that $n=rc=r(\tfrac{q-1}{3})$ is a primitive divisor of $q^r-1$. Notice that the statement $n\mid q^r-1$ is equivalent to $r\mid 3 \,\Psi_{r}(q)$, since $\Psi_{r}(q)=\frac{q^r-1}{q-1}$. By hypothesis $q\equiv 1\pmod {r}$, this implies that 
	$$\Psi_{r}(q)=\sum_{i=0}^{r-1}q^i\equiv r \equiv 0 \pmod{r}.$$
	Thus $n\mid q^r-1$ as desired. 
	
	It is enough to show that $n\nmid p^{l}-1$ for all $1\le l\le r-1$. 
	Assume first that $m\mid l$, i.e\@ $l=hm$ for some $1 \le h\le r-1$, 
	then the statement $n\nmid p^{hm}-1$ is equivalent to $r\nmid 3\,\Psi_{h}(q)$ and this is equivalent to $r\nmid \Psi_{h}(q)$ since $(3,r)=1$. 
	By hypothesis $q\equiv 1\pmod{r}$, then $\Psi_{h}(q)\equiv h \not\equiv 0\pmod{r}$, therefore $n \nmid p^{hm}-1$ for all $1\le h \le r-1$.
	
	On the other hand, if $l<m$ then $n$ cannot divide $p^l-1$, since $c$ divides $n$ and $c$ is a primitive divisor of $p^{m}-1$. On the other hand, if $ m\le l \le rm$ and $n\mid p^l-1$ we necessarily have that $m\mid l$. Indeed, if $l=md+e$ with $0\le e<m-1$ then 
	$$p^l \equiv p^e \pmod c.$$ 
	But $p^l \equiv 1 \pmod c$ since $c\mid n$. The primitive divisibility of $c$ implies that $e=0$, therefore $m\mid l$, that is $l=hm$ with $1\le h\le r-1$. By the last case $n \nmid p^{hm}-1$ for all $1\le h \le r-1$, therefore $n$ is a primitive divisor of $q^r-1$, as desired.

	The statement now follows from Theorem \ref{bebes2} proceeding as in the proofs of Propositions \ref{1-weight} and \ref{bebes3}.
\end{proof}

\begin{exam}
In the notation of the previous 
theorem, let $p=7$, $r=2$, $t=1$, $m=3t=3$, $q=p^3=343$ , and hence $c=\frac{q-1}{3}=114$. Clearly $p\equiv 1\pmod{3}$, $(r,3)=1$ and $q\equiv 1\pmod{r}$. 
In this case, it is not difficult to see that $a=b=1$ satisfying $4 \sqrt[3]{q}=a^2+27b^2$ with $(a,p)=1$ and  $a\equiv 1\pmod{3}$.

By the last theorem, the weights of the irreducible cyclic code $\CC(\frac{3(q+1)}{2},q^2)=\CC(516,7^6)$, after routine calculations, are given by
$$w_{\ell_1,\ell_2,\ell_3}=2\,(49\,h-\ell_1+5\ell_2-4\ell_3)$$ 
for $0\le h=\ell_1+\ell_2+\ell_3\le 2$ 
with $\ell_i$'s non-negative integers, with frequencies  $A_{\ell_1,\ell_2,\ell_3}$. 
By a simple analysis of cases, we obtain that the weight distribution of $\CC(516,7^6)$ is given by Table \ref{specejemplito2}.
\renewcommand{\arraystretch}{1.2}
\begin{table}[h!]
	\centering
	\caption{Weight distribution of $\CC(516,7^6)$.} \label{specejemplito2}
	\begin{tabular}{|c|c|}  
		\hline 
		weight & frequency \\ 
		\hline 
		$w_{0,0,0}=0$ & $A_{0,0,0}=1$ \\
		\hline
		$w_{1,0,0}=96$ & $A_{1,0,0}=228$  \\ 
		\hline
		$w_{0,1,0}=108$ & $A_{0,1,0}=228$  \\
		\hline
		$w_{0,0,1}=90$& $A_{0,0,1}=228$  \\
		\hline
		$w_{2,0,0}=192$& $A_{2,0,0}=114^2$  \\ 
		\hline
	\end{tabular} \:\:
	\begin{tabular}{|c|c|}
		\hline 
		weight & frequency \\ 
		\hline 
		$w_{0,2,0}=216$ & $A_{0,2,0}=114^2$  \\
		\hline
		$w_{0,0,2}=180$& $A_{0,0,2}=114^2$  \\ 
		\hline  
		$w_{1,1,0}=204$& $A_{1,1,0}=2\cdot 114^2$ \\
		\hline  
		$w_{1,0,1}=186$& $A_{1,0,1}=2\cdot 114^2$ \\
		\hline
		$w_{0,1,1}=198$& $A_{0,1,1}=2\cdot 114^2$ \\
		\hline
	\end{tabular}
\end{table}
	\hfill $\lozenge$
\end{exam}

Proceeding similarly as in the proof of the previous theorem, one can obtain the weight distribution of irreducible cyclic codes obtained from $\CC(4,q)$. We leave the details to the reader.

\begin{thm} \label{c4}
	Let $p,$ and $r$ be different primes with  $p\equiv 1 \pmod 4$ and let $c,k,m,q,t$ be integers such that $m=4t$, $q=p^m$, $c=\frac{q-1}{4}$ and $k=\frac{4}{r}\Psi_{r}(q)$. If $q\equiv 1\pmod{r}$ and $(4,r)=1$, then the weights of $\CC(k,q^r)$ are given by 
	$$w_{\ell_1,\ell_2,\ell_3,\ell_4} = \tfrac{p-1}{4p} \{ hq+(\ell_1+\ell_2-\ell_3-\ell_4)\sqrt{q}+(2a(\ell_1-\ell_2)+4b(\ell_3-\ell_4))p^t \} $$ 
	for $0\le h=\ell_1+\ell_2+\ell_3+\ell_4 \le r$ with $(\ell_1,\ell_2,\ell_3,\ell_4)\in \mathbb{N}_{0}^4$,
	where $a,b$ are the unique integers satisfying 
	 $$ \sqrt{q}=a^2+4b^2, \qquad a\equiv 1 \pmod 4 \qquad \text{and} \qquad (a,p)=1.$$
	 with frequencies 
	 $$A_{\ell_1,\ell_2,\ell_3,\ell_4} = \tbinom{r}{h} \tbinom{h}{\ell_1,\ell_2,\ell_3,\ell_4} c^h .$$ 
\end{thm}

\begin{rem}
The condition $k\mid \frac{q-1}{p-1}$, which allows us to switch between the spectrum of the graph $\Gamma(k,q)$ and the weight distribution of the code $\CC(k,q)$, implies that $p\equiv \pm 1 \pmod k$ for $k=3,4$. The cases not covered by Theorems \ref{c3} and \ref{c4}, that is $p\equiv -1 \pmod k$ with $k=3,4$, are semiprimitive ones and fall into the case of Theorem \ref{bebes3}.
\end{rem}

\section{Number of rational points of Artin-Schreier curves}
In this section we consider Artin-Schreier curves $C_{k,\beta}(p^m)$ with affine equations
\begin{gather} \label{AS curve} 
C_{k,\beta}(p^m)  : \qquad  y^{p}-y = \beta x^{k}, \quad  \beta \in \ff_{p^m} 
\end{gather}
with $k \mid p^m-1$. 
A good treatment of Artin-Schreier curves can be found in Chapter~3 by G\"uneri-\"Ozbudak in \cite{GS}. 

We begin by establishing a direct relationship between the number of rational points of $C_{k,\beta}(p^m)$ and the eigenvalue 
$\lambda_{\beta}$ of $\G(k,p^m)$  --see equation \eqref{lambdagamma}--.
\begin{prop} \label{curve and graph}
	Let $p$ be a prime and let $k,n, m$ be positive integers such that $k\mid\frac{p^m-1}{p-1}$ and $n=\frac{p^m-1}{k}$. If $n$ is a primitive divisor of $p^m-1$ then
	\begin{equation}\label{ratpointeigenval}
	\#C_{k,\beta}(p^m) = 2p^{m}+ k(p-1)\lambda_{\beta} 
	\end{equation}
	for all $\beta\in \ff_{p^m}$.
\end{prop}

\begin{proof}
	The code $\CC_k = \{ c_{k}(\beta)=(\Tr_{p^m/p}(\beta x^k))_{x\in \ff_{p^m}^*} : \beta \in \ff_{p^m}\}$
	is obtained from $k$-copies of $\CC(k,p^m)$. This implies that 
	\begin{equation*}
	w(c_{k}(\beta)) = k \, w(c(\beta)) \qquad \text{where} \qquad 
	c_k(\beta) = \big( \Tr_{p^m/p}(\beta \omega^{ik}) \big)_{i=1}^n.
	\end{equation*}
	
	On the other hand, the weight of the codeword $c_k(\beta)$ is related to the number of  $\mathbb{F}_{p^m}$-rational points of the curve $C_{k,\beta}(p^m)$.
	In fact, by Theorem 90 of Hilbert we have  
	$$ \Tr_{p^m/p}(\beta x^k) = 0 \qquad \Leftrightarrow \qquad y^p-y = \beta x^k \quad \text{ for some } y \in \ff_{p^m}.$$  
	Since $C_{k,\beta}(p^m)$ is a $p$-covering of $\mathbb{P}^1$, considering the point at infinity, we get 
	\begin{equation*} \label{rat points}
	\#C_{k,\beta}({p^m})  =  1 + p \, \#\{ x\in \ff_{p^m}: \Tr_{p^m/p}(\beta x^k) = 0 \}  = p^{m+1} - p \, w(c_k(\beta)) +1.
	\end{equation*}
	Then, equation \eqref{ratpointeigenval} follows directly from the last equality and \eqref{Pesoaut}.
\end{proof}

\subsubsection*{Artin-Schreier curves over extensions}
As an application of Theorem \ref{bebes2}, we will next obtain a relationship between the rational points of Artin-Schreier curves as in \eqref{AS curve} defined over two different fields 
$$\ff_{p^a} \subset \ff_{p^m},$$ 
with $p$ a fixed prime. 
We recall the notation 
$\Psi_b(x) = \tfrac{x^b-1}{x-1} = x^{b-1} + \cdots + x^2+x+1$.

\begin{coro} \label{ratpoints}
Let $p$ be a prime and let $k, m=ab, n, a, b, c, u$ as in Theorem \ref{bebes2}.
Then, for each $\beta\in \ff_{p^m}$ there are $\alpha_1,\ldots,\alpha_{b}\in \ff_{p^{a}}$ such that 
\begin{equation} \label{relation rm} 
	\#C_{k,\beta}({p^m}) = 
	\tfrac 1b \Psi_{b}(p^a) \sum_{i=1}^b \#C_{u,\alpha_i}({p^{a}}) - (p+1)p^a \Psi_{b-1}(p^a).
	\end{equation}
	Conversely, given $\alpha_1,\ldots,\alpha_{b}\in \ff_{p^{a}}$ there exists $\beta\in\ff_{p^m}$ satisfying \eqref{relation rm}.
\end{coro}

\begin{proof}
Consider the cyclic codes $\CC_k$ and $\CC(k,p^m)$ as before and the analogous ones $\CC_u$ and $\CC(u,p^a)$. 
	Proceeding similarly as in the the proof of Proposition \ref{curve and graph}, we have that 
	\begin{equation*} \label{rat points2}
	\#C_{u,\alpha}(\mathbb{F}_{p^{a}})  =  1 + p \cdot \#\{ x\in \ff_{p^{a}}: \Tr_{p^{a}/p}(\alpha x^u) = 0 \}  = 
	p^{{a}+1} - p \, w(c_u(\alpha)) +1.
	\end{equation*}
	
	First notice that $\ff_{p^{a}} \subset \ff_{p^m}$. Now, by Theorem \ref{bebes2}, for each element $\beta\in\ff_{p^m}$ there exist elements 
	$\alpha_1,\ldots,\alpha_b \in \ff_{p^{a}}$ such that
	$w(c(\beta)) =  
	w(c(\alpha_1)) + \cdots + w(c(\alpha_b))$. 
	Moreover, given $\alpha_1,\ldots,\alpha_b \in \ff_{p^{a}}$, 
	$w(c(\alpha_1)) + \cdots + w(c(\alpha_b))$ defines a weight in $\CC(k,p^m)$, i.e.\@ there must be some 
	$\beta\in\ff_{p^m}$ such that 
	$$w(c(\beta)) = 
	w(c(\alpha_1)) + \cdots + w(c(\alpha_b)).$$ 
	Therefore, the number $\#C_{k,\beta}({p^m})$ equals
	\begin{equation} \label{Ckpm alphas} 
	p^{m+1} +1 - pk\sum_{i=1}^b w(c(\alpha_i)) 
	=  p^{m+1} +1 - \tfrac{k}{u}\sum_{i=1}^b(p^{{a}+1} +1 -\#C_{u,\alpha_i}({p^{a}})).
	\end{equation}
	Since $\frac{k}{u} = \frac{p^m-1}{b(p^{a}-1)} = \tfrac 1b \Psi_b(p^a)$, after straightforward calculations we get \eqref{relation rm} 
	as desired.
\end{proof}

In particular, from \eqref{relation rm} we have 
$$ \#C_{k,\beta}({p^m}) \equiv \tfrac 1b \Psi_{b}(p^a) \sum_{i=1}^b \#C_{u,\alpha_i}({p^{a}}) \pmod M$$
with $M=p+1$, $M=p^a$ or $\Psi_{b-1}(p^a)$. Since $\Psi_{t+1}(x) = x^t +\Psi_{t}(x)$, taking $x=p^a$ and $t=b-1$ we also have
$$ b \cdot \#C_{k,\beta}({p^m}) \equiv 
p^{a(b-1)} \sum\limits_{i=1}^b \#C_{u,\alpha_i}({p^{a}}) \quad \pmod{\Psi_{b-1}(p^a)}.$$

\begin{exam} \label{ejemplin}
	In the notations of Theorem \ref{bebes2}, take $p=2$ and $u=1$. Hence, $c=2^a-1$, $n=b(2^{a}-1)$ and $m=ab$. Obviously $2^{a}-1$ is a primitive divisor of itself and it can be shown that if $b$ is odd and $x=2^a \equiv 1 \pmod b$ then $n$ is a primitive 
	divisor of $2^m-1$.  If $k = \Psi_b(x)$, by the last corollary the $\ff_{2^m}$-rational points of the curve 
	\begin{equation} \label{Ckb2m}
    C_{k,\beta}(2^m) :\quad  y^{2}+y = \beta x^{k}
	\end{equation}
	with $\beta\in\ff_{2^m}$ can be calculated in terms of the $\ff_{2^{a}}$-rational points of the curves 
	$$C_{1,\alpha_i}(2^a): \quad  y^{2}+y = \alpha_{i} x$$ 
	for some $\alpha_1,\ldots,\alpha_{b}\in \ff_{2^{a}}$. 
	
	The simplex code $\CC(1,2^a)$ has only one nonzero weight, which is $2^{a-1}$. Taking into account that 
	$$\#C_{1,\alpha}({2^a}) = 2^{a+1}-2w(c_{1}(\alpha))+1$$ 
	with $w(c_{1}(\alpha))\in \CC(1,2^a)$ we have that 
	$\#C_{1,\alpha}({2^a}) = 2^{a+1}+1$ or $2^{a}+1$ for all $\alpha\in\ff_{2^a}$. By \eqref{Ckpm alphas} and Corollary \ref{ratpoints}, we have that the number of $\ff_{q^m}$-rational points of each curve in \eqref{Ckb2m}  
	is given by $2^{m+1} +1 - k\ell 2^{a}$ for some $\ell$ depending on $\beta$ ranging over all the interval $0\le\ell\le b$, that is
	$$\{ \#C_{k,\beta}(2^m)\}_{\beta\in\ff_{2^m}} = \{2^{m+1} +1 - k\ell 2^{a} : 0\le\ell\le b \}.$$	%\hfill $\lozenge$
\end{exam}

\section{A reduction formula for Gaussian periods}
The Gaussian periods $\eta_{i}^{(N,q)}$ defined in \eqref{gaussian} satisfy some arithmetic relations.
From Theorem 14 in \cite{DY}, we have the following integrality results:
\begin{equation} \label{int gp}
\eta_i^{(N,q)} \in \Z \qquad \text{and} \qquad N \eta_i^{(N,q)} +1 \equiv 0 \pmod p
\end{equation}
where $q=p^m$ and $N=gcd(\frac{q-1}{p-1},k)$. 
Furthermore, if $k\mid \frac{q-1}{p-1}$ then $N=k$ and we have   
\begin{equation} \label{cond etas}
\sum_{i=0}^{k-1} \eta_i^{(k,q)} = -1 \qquad \text{and} \qquad 
\sum_{i=0}^{k-1} \eta_{i}^{(k,q)} \eta_{i+j}^{(k,q)} = q \theta_j - n \quad (0\le j \le k-1)
\end{equation}
with $n=\frac{q-1}k$ and where $\theta_j=1$ if and only if $-1\in C_j^{(k,q)}$ and $\theta_j=0$ otherwise (see \cite{St}). 
Equivalently, $\theta_j=1$ if and only if either $n$ is even and $j=0$ or else $n$ is odd and $j=\frac k2$.
Apart from \eqref{int gp} and \eqref{cond etas}, there are not many known relations for Gaussian periods (to our best knowledge).

As another application of Theorem \ref{bebes2}, we next give a relation between Gaussian periods defined over 
two different fields
$\ff_{p^a} \subset \ff_{p^m}$, showing that one can reduce the computation of 
$\eta_i^{(k,q)}$ to integral linear combinations of Gaussian periods $\eta_j^{(u,a)}$ with smaller parameters, namely $u\mid k$ and $a \mid m$.

\begin{prop} \label{gaussitos}
Let $p$ be a prime and let $k,m,n,a,b,c,u$ be integers as in Theorem \ref{bebes2}. Then, for each $i=0, \ldots,k-1$ there exist  
	integers $s \in \N$ and 
	$\ell_0,\ell_1,\ldots,\ell_{s} \in \N_0$ 
	such that  
	\begin{equation} \label{eta kpm}
	\eta_{i}^{(k,p^m)} = c\ell_0 + \sum_{j=1}^{s} \ell_{j} \eta_{i_j}^{(u,p^{a})}
	\end{equation}
	where the $\ell_i$'s run over all possible $(s+1)$-tuples $(\ell_0,\ldots,\ell_{s})$ such that 
	$\ell_0 + \cdots + \ell_{s} = b$ different from $(b,0,\ldots,0)$. 
\end{prop}

\begin{proof}
	By hypothesis, $\G=\G(k,p^m)$ decomposes as $\G=\square^{b}\Gamma_0$ where $\G_0=\G(u,p^{a})$. 
	Let $q=p^m$ and $z=p^{a}$. 
	We know that the spectra of $\Gamma$ and $\Gamma_0$ are given in terms of Gaussian periods. 
	In fact, by Theorem 2.1 in \cite{PV3} we have that 
	\begin{equation} \label{specs}
	\begin{split}
	Spec(\Gamma(k,q)) = \{ \Lambda_0= n, \quad \Lambda_1= \eta_1^{(k,q)}, \quad  \ldots, \quad \Lambda_{k-1} = \eta_{k-1}^{(k,q)} \}, \\
	Spec(\Gamma(u,z)) = \{\lambda_0= c, \quad \lambda_1 = \eta_1^{(u,z)}, \quad \ldots, \quad \lambda_{u-1} = \eta_{u-1}^{(u,z)} \}.
	\end{split}
	\end{equation}
	
	By \eqref{spec sum} in the proof of Theorem \ref{bebes2},
	the eigenvalues of $\G$ and of $\G_0$ are related by the expression
	\begin{equation} \label{rel eigs}
	\Lambda_{\ell_0,\ldots,\ell_{s}} = \ell_0 \lambda_0 + \cdots + \ell_{s} \lambda_{s}	
	\end{equation}
	where $\ell_0, \ldots, \ell_{s}$ are integers satisfying 
	$\ell_0 +\cdots + \ell_{s}=b$. 
	By \eqref{specs} and \eqref{rel eigs} we have \eqref{eta kpm}. 
	It remains to rule out all the cases giving 
	$$ \eta_i^{(k,q)}=n=bc.$$
	But the only way to have $\eta_i^{(k,q)}=n$ is given by 
	$(\ell_0,\ell_1,\ldots,\ell_{s}) = (b,0,\ldots,0)$, since $\ell_0 + \cdots + \ell_{s}=b$, and the result thus follows.  
\end{proof}

\begin{rem}
The Gaussian periods $\eta_0^{(k,q)}, \ldots, \eta_{k-1}^{(k,q)}$ with $(k,q)$ a semiprimitive pair are explicitly known 
(see Lemma 13 in \cite{DY}).
\end{rem}

We now show that if $\G=\G(k,p^m)$ is Cartesian decomposable, say $\G \simeq \square^b \G_0$, with $\G_0=\G_0(u,p^a)$ a semiprimitive GP-graph then we can explicitly compute the Gaussian periods $\eta_i^{(k,p^m)}$. 

\begin{prop} \label{GP upr}
	Let $q=p^m$ with $p$ prime and $k\mid q-1$ such that $n=\frac{q-1}k=bc$ where $m=ab$, $u=\frac{p^a-1}{c}$ and $(u,p^a)$ is a semiprimitive pair. Then, the different Gaussian periods modulo $q$ are given by
	\begin{equation} \label{eta upr}
	\eta_i^{(k,q)} = \ell_0 c + \ell_1 \tfrac{(u-1) \sigma \sqrt{p^a} -1}{u} - \ell_2  \tfrac{\sigma \sqrt{p^a}+1}{u} 
	\end{equation}
	where the non-negative integers $\ell_0, \ell_1, \ell_2$ run in the set 
	$$\{(\ell_0, \ell_1, \ell_2) : \ell_0+\ell_1+\ell_2=b\} \smallsetminus \{(b,0,0)\}$$ and 
	$\sigma = (-1)^{\frac{m}{2t}+1}$ with $t$ the least integer $j$ such that $u \mid p^j+1$.
\end{prop}

\begin{proof}
	By Corollary \ref{gaussitos} we have an expression for each $\eta_i^{(k,q)}$ in terms of the $\eta_j^{(u,p^a)}$'s. Since $(u,p^a)$ is a semiprimitive pair, there are only two different such periods, given by (3.4) and (3.5) of \cite{PV3}, depending the case. 
	In case $(a)$, that is $p, \alpha$ and $s$ odd, we have 
	$$\eta_0^{(u,p^a)} = \tfrac{(u-1)\sqrt{p^a}-1}{u} \qquad \text{and} \qquad \eta_1^{(u,p^a)}= -\tfrac{\sqrt{p^a}+1}{u}$$ 
	while in case $(b)$ we have 
	$$\eta_0^{(u,p^a)} = -\tfrac{\sigma\sqrt{p^a}+1}{u} \qquad \text{and} \qquad 
	\eta_1^{(u,p^a)} = \tfrac{\sigma (u-1)\sqrt{p^a}-1}{u}.$$
	Now, by \eqref{eta kpm} we get 
	$$\eta_i^{} = \ell_0 \, c + \ell_1 \, \eta_0^{(u,p^a)} + \ell_2 \, \eta_1^{(u,p^a)} .$$ 
	Since the triples $(\ell_0,\ell_1,\ell_2)$ satisfying $\ell_0+\ell_1+\ell_2=b$ are symmetric, the above expression is the same no matter if we are in case of $(a)$ or $(b)$, or if $\sigma$ is $1$ or $-1$, and hence we get \eqref{eta upr}.
\end{proof}

\begin{exam}
	Take $u=2$, $a=2$, $b=3$ and $p=5$. Then $(u,p^a)= (2,5^2)$ is a semiprimitive pair and $\G_0=\G(2,5^2)=P(25)$, a classic Paley graph. Thus, we have $m=ab=6$, $q=5^6=15{.}625$, $c=\frac{p^a-1}u=\tfrac{5^2-1}2=12$ and $n=bc=36$; hence $k=\tfrac{q-1}n=434$. 
	
	By \eqref{eta upr}, the Gaussian periods for $(k,q)=(434,15{.}625)$ are given by
	$$\eta_i^{(434,15{.}625)} = 12 \ell_0 + 2\ell_1-3\ell_2$$ 
	where $\ell_0+\ell_1+\ell_2=3$ and $(\ell_0,\ell_1,\ell_2)\ne (3,0,0)$; 
	compare with \eqref{gaussian}.
	There are 9 such triples, 
	namely	$(2,1,0)$, $(2,0,1)$, $(1,2,0)$, $(1,1,1)$, $(1,0,2)$, 
	$(0,3,0)$, $(0,2,1)$, $(0,1,2)$ and $(0,0,3)$. 
	Thus, we have that
	$$\eta_1 =26, \: \eta_2=21, \: \eta_3=16, \:\:\eta_4=11, \:\:
	\eta_5 =\eta_6 =6, \:\:\eta_7 =1, \:\:	\eta_8 =-4, \:\:\eta_9 =-9.$$
	Note that $\eta_i \equiv 1 \pmod 5$ for $1\le i \le 9$ as it should be, since $k\eta_i \equiv -1 \pmod p$ by \eqref{int gp}. 
	
	We now check the expressions in \eqref{cond etas}.
	If $\eta_i^{(k,q)}$ is associated with $(\ell_0,\ell_1,\ell_2)$, then its frequency is given by $\mu_i=\tfrac 1n A_i$  
	where 
	$$A_i=A_{\ell_0,\ell_1,\ell_2} = \tbinom{3}{\ell_0, \ell_1, \ell_2} m_0^{\ell_0} m_1^{\ell_1} m_2^{\ell_2},$$ 
	with 
	$m_0, m_1, m_2$ the multiplicities of the Paley graph $P(25)$. The spectrum of $P(q)$ is well-known and it is given by
	$$Spec(P(p^2)) = \{[\tfrac{p^2-1}2]^1, [\tfrac{p-1}2]^{n}, [\tfrac{-p-1}2]^{n}\}$$
with $n=\frac{p^2-1}{2}$. Hence, $Spec(P(25)) = \{[12]^1, [2]^{12}, [-3]^{12}\}$ and we thus have 
$m_0=1$ and $m_1=m_2=12$. 
	In this way we obtain 
	\begin{gather*}
	 A_{2,1,0}=A_{2,0,1}=3 \cdot 12=36,  \qquad  A_{1,2,0}=A_{1,0,2}=3\cdot 12^2=432, \\ 
	 A_{1,1,1}=6\cdot 12^2=684, \qquad   A_{0,2,1}=A_{0,1,2}=3\cdot 12^3=5184,  \\ 
	 A_{0,3,0}=A_{0,0,3}=1\cdot 12^3=1728,
	\end{gather*}
	and hence 
	$$\mu_1=\mu_2=1, \quad \mu_3=\mu_5=12, \quad \mu_4=24, \quad  \mu_6=\mu_9=48, \quad \mu_7=\mu_8= 144.$$
	Therefore we have
	$$\sum_{i=0}^{433} \eta_i^{(434,5^6)}  = \sum_{i=1}^9 \mu_i \eta_i$$
	and hence 
	\begin{eqnarray*}
		\sum_{i=0}^{433} \eta_i^{(434,5^6)} 
		&=&	\mu_1(\eta_1 + \eta_2) + \mu_3(\eta_3 + \eta_5) + \mu_4 \eta_4  + \mu_6(\eta_6 + \eta_9) + \mu_7 (\eta_7 + \eta_8) \\
		& = & (26+21) + 12(16+6) + 24 \cdot 11 + 48(6-9) + 144(1-4) = -1.
	\end{eqnarray*} 
	One can also check that 
	$$\sum_{i=1}^9 \mu_i \eta_i^2 = 15{.}589 = q-n \qquad \text{and} \qquad 
	\sum_{i=1}^9 \mu_i \mu_{i+j} \eta_i \eta_{i+j} = -36 = -n$$ for 
	$j=1,\ldots,9$, and hence the second identity of \eqref{cond etas} holds.
	\hfill $\lozenge$
\end{exam}

\end{document}